\documentclass[12pt]{amsart}
\usepackage{amsfonts,latexsym,rawfonts,amsmath,amssymb,amsthm,mathrsfs}
\usepackage[plainpages=false]{hyperref}
\usepackage{graphicx}
\usepackage{comment}
\usepackage[pagewise]{lineno}

\numberwithin{equation}{section}

\RequirePackage{color}
\theoremstyle{plain}

\newtheorem{theorem}{Theorem}[section]
\newtheorem{lem}[theorem]{Lemma}

\theoremstyle{Corollary}
\newtheorem{cor}[theorem]{Corollary}

\newtheorem{proposition}{Proposition}[section]
\newtheorem{remark}{Remark}[section]
\newcommand{\beq}{\begin{equation}}
\newcommand{\eeq}{\end{equation}}
\newcommand{\beqs}{\begin{eqnarray*}}
\newcommand{\eeqs}{\end{eqnarray*}}
\newcommand{\beqn}{\begin{eqnarray}}
\newcommand{\eeqn}{\end{eqnarray}}
\newcommand{\beqa}{\begin{array}}
\newcommand{\eeqa}{\end{array}}

\begin{document}
\title{An Anisotropic shrinking flow and $L_p$ Minkowski problem }

\author{Weimin Sheng}
\address{Weimin Sheng: School of Mathematical Sciences, Zhejiang University, Hangzhou 310027, China.}
\email{weimins@zju.edu.cn}

\author{Caihong Yi}
\address{Caihong Yi:
School of Mathematical Sciences, Zhejiang University, Hangzhou 310027, China.}
\email{11735001@zju.edu.cn}

\thanks{The authors were supported by NSFC, grant nos. 11971424 and 11571304.}


\begin{abstract}
In this paper, we consider a shrinking flow of smooth, closed, uniformly convex hypersurfaces in Euclidean $R^{n+1}$ with speed $fu^\alpha\sigma_n^{-\beta}$, where $u$ is the support function of the hypersurface, $\alpha, \beta \in R^1$, and $\beta>0$, $\sigma_n$ is the $n$-th symmetric polynomial of the principle curvature radii of the hypersurface.  We prove that the flow exists a unique smooth solution for all time and converges smoothly after normalisation to a smooth solution of the equation
$fu^{\alpha-1}\sigma_n^{-\beta}=c$ in the following cases $1-n\beta-2\beta<\alpha<1+n\beta$, $\alpha\ne1-\beta$ and $\alpha=0$, $\beta=1$ respectively, provided the initial hypersuface is origin-symmetric and $f$ is a smooth positive even function on $S^n$.  For the case $\alpha\ge 1+n\beta$, $\beta>0$, we prove that the flow  converges smoothly after normalisation to a unique smooth solution of
$fu^{\alpha-1}\sigma_n^{-\beta}=c$ without any constraint on the initial hypersuface and smooth positive function $f$. When $\beta=1$, our argument provides a uniform proof to the existence of the solutions to the $L_p$ Minkowski problem $u^{1-p}\sigma_n=\phi$ for $p\in(-n-1,+\infty)$ where $\phi$ is a smooth positive function on $S^n$.
\end{abstract}

\maketitle

\baselineskip16pt
\parskip3pt

\section{Introduction}
Let $\mathcal{M}_0$ be a smooth, closed and uniformly convex hypersurface in $R^{n+1}$, and $\mathcal{M}_0$ encloses the origin. We study the following anisotropic shrinking curvature flow
\begin{equation}\label{SF}
\left\{
\begin{array}{ll}
\frac{\partial{X}}{\partial{t}}(\cdot,t)&=-f(\nu)< X,\nu>^{{\alpha}}K^{\beta}\nu,\\
X(\cdot,0)&=X_0(\cdot),
\end{array}
\right.
\end{equation}
where $\mathcal{M}_t$ is parametrized by the inverse Gauss map $X:S^n\to \mathcal{M}_t\subset R^{n+1}$ and encloses origin, $K$ is the Gauss curvature of $\mathcal{M}_t$, $\nu$ is the unit outer normal at $X(\cdot,t)$, and $f$ is a smooth positive function on $S^n$.

In 1974, Firey \cite{Fir74} firstly introduced the Gauss curvature flow as a model for the shape change of tumbling stones. Huisken \cite{Hui84} considered the mean curvature flow in 1984.
Thereafter, a range of flows with the speed of the symmetric polynomial of principal curvatures were studied, see \cite{Chow85, Chow87, And99, And00} etc.   For the curvature flow at the speed of $\alpha$-power of the Gauss-Knonecker curvature, in the affine invariant case $\alpha=\frac{1}{n+2}$, Andrews\cite{And96} showed that the flow converges to an ellipsoid. It was conjectured that the solution will converge to a round point along the flow for $\alpha>\frac{1}{n+2}$. Chow \cite{Chow85}, Andrews \cite{And94},  Andrews et al.\cite{AGN16}, Choi and Daskalopoulos \cite{ChoiDask16}, gave some partial answers respectively. In \cite{BCD16}, Brendle et al. finally resolved the conjecture for all $\alpha>\frac{1}{n+2}$ in all dimensions recently. As a natural extension, anisotropic flows usually provide alternative proofs and smooth category approach of the existence of solutions to elliptic PDEs arising in convex geometry, see \cite{Wang96, And97, ChWang00, GuNi17, LSW16, Iva2018} etc.. For the existence problem of the prescribed polynomial of the principal curvature radii of the hypersurface, Urbas\cite{Urb91}, Chow and Tsai\cite{ChowH97}, Gerhardt \cite{Gerh14}, Xia \cite{X16}, Li, Sheng and Wang\cite{LSW18} studied the convergence of the flows with the speed of $F(\lambda_1,\ldots,\lambda_n)$, where $F$ is a certain symmetric polynomial of the principal curvature radii $\lambda_1,\ldots, \lambda_n$ of the hypersurface. Especially in \cite{And97}, Andrews studied an anisotropic shrinking flow. By introducing some monotone quantities, he proved the flow converges after normalisation to a smooth hypersurface which satisfies a soliton equation.

Under the flow \eqref{SF}, the support function $u$ satisfies
\begin{equation}\label{SF1}
\left\{
\begin{array}{ll}
\frac{\partial{u}}{\partial{t}}(x,t)&=-f(x)u^{{\alpha}}(x,t){\sigma}_{n}^{-\beta},\\
u(\cdot,0)&=u_0(\cdot).
\end{array}
\right.
\end{equation}
where $\sigma_n$ is the $n$-th elementary symmetric function for principal curvature radii, i.e.
\[\sigma_n(.,t)=\lambda_{1}\cdots\lambda_{n},\]
$\lambda_i \, (1\le i\le n)$ is the principal curvature radii of hypersurface $\mathcal{M}_t$. We prove that the flow exists for all time and converges smoothly after normalisation to a soliton which is a solution of $fu^{\alpha-1}\sigma_n^{-\beta}=c$  in the following cases:
 $1-n\beta-2\beta<\alpha<1+n\beta$, $\alpha\ne1-\beta$ and $\alpha=0$, $\beta=1$, respectively,
if the initial hypersurface is origin-symmetric and $f$ is a smooth positive even function on $S^n$. For the case $\alpha\ge1+n\beta$, $\beta>0$, we prove that the flow converges smoothly after normalisation to a unique smooth solution of
$fu^{\alpha-1}\sigma_n^{-\beta}=c$ without any constraint on the initial hypersurface and the smooth positive function $f$.

In fact, when $\beta=1$, the elliptic equation $fu^{\alpha-1}\sigma_n^{-\beta}=c$ is just the well-known $L_p$ Minkowski problem $u^{1-p}\sigma_n=\phi$ for $p\ge -n-1$ in the smooth category. The $L_p$ Minkowski problem was introduced by Lutwak in \cite{Luk93}, where he asked for necessary and sufficient conditions that would guarantee that a given measure on the unit sphere would be the $L_p$ surface area measure of a convex body.  Our proof provides a uniform approach to the existence of the solutions to the problem for the case $-n-1< p<n+1$ with the assumption that the function $\phi$ is even, and the case $p\ge n+1$ without any constraint on $\phi$.  In \cite{Luk93} Lutwak proved the solution to the $L_p$ Minkowski problem is unique for $p>1$ and $p\neq n$ if $\phi$ is an even positive function.  In \cite{LO95}Lutwak and Oliker also proved the regularity of the solution in this case. When $p=-n-1$, it is the centro-affine Minkowski problem which was studied by Chou-Wang \cite{ChWang06}, Lu-Wang \cite{LW13}, Zhu \cite{Zhu15} and Li \cite{L17}. In \cite{ChWang06} the authors also considered the $L_p$ Minkowski problem without the evenness assumption on $\phi$, and proved the existence of the $C^2$ convex solution for the case $p\ge1+n$ and the weak solution for the case $1<p<n+1$. The uniqueness of the solution was also proved for $p>n+1$ in \cite{ChWang06}. When $p=1$, it is the classical Minkowski problem, it was finally solved by Cheng-Yau\cite{CY77} and Pogorelev\cite{P78}. For the case $0\le p<1$, Haberl et al. \cite{CEDZ10}, Zhu \cite{Zhu15} studied the existence of the solutions, and Chen et al.\cite{CLZ17} finally solved the problem. Jian et al. \cite{JLW15} proved that the $L_p$ Minkowski problem admits two solutions when $-n-1<p<0$.  Y. He et al. \cite{HLW16} constructed multiple solutions for the case $-n-1<p<-n$. The additional extensions for  $L_p$ Minkowski problem can be learned, see, \cite{HLYZ05, C06, BLDZ13, HLX15} etc. for example. By constructing an anisotropic expanding flow, Bryan et al. \cite{BIS16} also gave a unified flow approach to the existence of smooth, even $L_p$ Minkowski problems for $p>-n-1$. Their approach is in $C^1$ when $p>n+1$, and for a subsequence when $p\in (-n-1, 1)$. Our theorem will improve their result.

We define
\[\widetilde{u}=\Big(\frac{|S^n|}{V_{n+1}(\underbrace{u,u,\ldots,u}_{(n+1)-times})}\Big)^{\frac{1}{n+1}}u,\]
where the definition of $V_{n+1}(u,u,\ldots,u)$ may refer Section 2. In fact, it is just the volume of convex body $\Omega_t$, where $\partial\Omega_t=\mathcal{M}_t$. A direct calculation shows
 \begin{eqnarray}\label{equality} \int_{S^n}\widetilde{u}\sigma_n[W_{\widetilde{u}}]d\mu=|S^n|.
\end{eqnarray}
Considering the following normalised flow of \eqref{SF1}
\begin{equation}\label{NSF1}
\left\{
\begin{array}{ll}
\partial_\tau u&=-fu^\alpha\sigma_n^{-\beta}+u\frac{\int_{S^n}fu^\alpha\sigma_n^{1-\beta}d x_{S^n}}{|S^n|},\\
u(.,0)&=u_0.
\end{array}
\right.
\end{equation}
where we still use $u$ instead of $\widetilde{u}$ for convenience, and
\[\tau=\int_0^t\Big(\frac{|S^n|}{V_{n+1}(u,u,\ldots,u)}\Big)^{\frac{1+n\beta-\alpha}{n+1}}ds.\]

We still use $t$ instead of $\tau$ to denote the time variable if no confusions arise, and we set
\begin{eqnarray}\label{def eta t}
\eta(t)=\frac{\int_{S^n}fu^\alpha\sigma_n^{1-\beta}dx}{|S^n|},
\end{eqnarray}
hence the flow \eqref{NSF1} can be written as
\begin{equation}\label{NSF}
\left\{
\begin{array}{ll}
\partial_t u&=-fu^\alpha\sigma_n^{-\beta}+\eta(t)u,\\
u(.,0)&=u_0.
\end{array}
\right.
\end{equation}
Now we introduce a quantity which is similar to the one introduced by Andrews in \cite{And97},
\[\mathcal{Z}_p(u(\cdot,t))=\int_{S^n}u\sigma_n(fu^{\alpha-1}\sigma_n^{-\beta})^pdx,\]
where $p\in R^1$. When $p=0$, $\mathcal{Z}_0(u(\cdot,t))=\int_{S^n} u\sigma_n dx=|S^n|$, see \eqref{equality}. We will show the quality $\mathcal{Z}_p(u(\cdot,t))$ plays a key role in this paper.

When $p=\frac{1}{\beta}$, consider the following functional

\begin{equation}\label{functional}
\mathcal{J}(u(\cdot,t))=\left\{\begin{array}{ll}
\mathcal{Z}_{\frac{1}{\beta}}(u(\cdot,t)),
&  \textrm{if $\alpha>1-n\beta-2\beta$, $\alpha\ne1-\beta$, $\beta>0$,}\\
\frac{\int_{S^n}f \log u dx}{\int_{S^n}f dx}-\frac{1}{n+1}\log\int_{S^n}u\sigma_n dx, &
\textrm{if $\alpha=0$, $\beta=1$.}
\end{array}\right.
\end{equation}
where the last functional were introduced by Huang et al. \cite{HLYZ16}.
We will show in Lemma \ref{monotone12}, Lemma \ref{monotone3} and Lemma \ref{monotone4} that $\mathcal{J}(u(\cdot,t))$ is strictly monotone along the flow \eqref{NSF} and $\frac{d}{dt}\mathcal{J}(u(\cdot,t))=0$ if and only if $u(\cdot,t)$ solves
\begin{equation}\label{ellipic t}
fu^{\alpha-1}\sigma_n^{-\beta}=\eta(t).
\end{equation}

\begin{figure*}[ht]
\centering
\includegraphics[width=7cm,height=7cm]{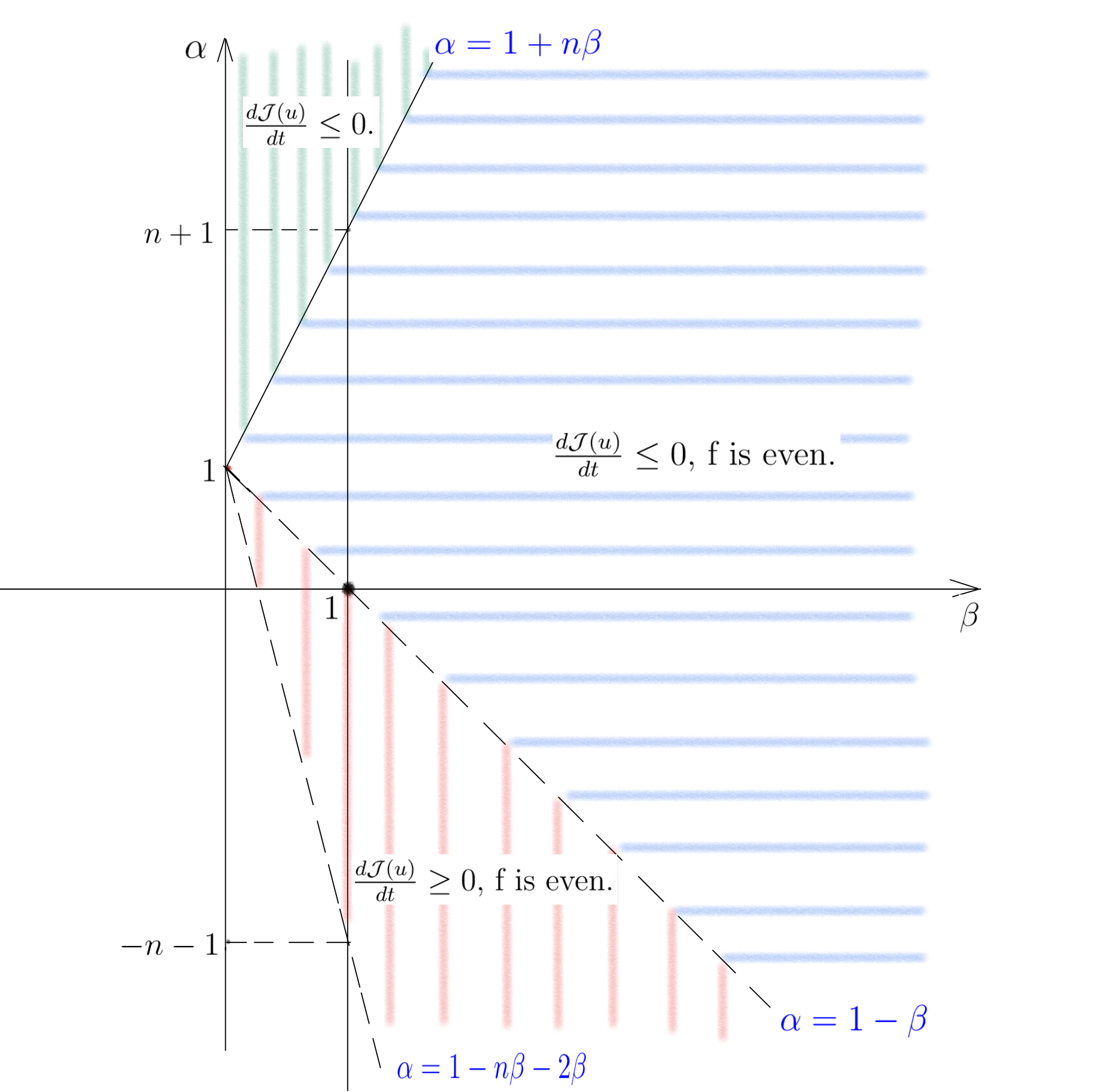}
\end{figure*}The monotonicity of the functional ensures that the normalised flow \eqref{NSF} converges to the elliptic equation
\begin{equation}\label{elliptic eq}
fu^{\alpha-1}\sigma_n^{-\beta}=c,
\end{equation}
for some positive constant $c$ as $t\to\infty$. When $\alpha\ne1+n\beta$,  if \eqref{elliptic eq} has a uniformly convex solution $u$, then $c^\frac{1}{1+n\beta-\alpha}u$ is just a solution of elliptic equation of $fu^{\alpha-1}\sigma_n^{-\beta}=1$ by homogeneity.
Note that when $\alpha=1-\beta$, the elliptic equation becomes $fu^{-\beta}\sigma_n^{-\beta}=1$ which is the equation $\bar{f}u^{-1}\sigma_n^{-1}=1$ with $\bar{f}=f^\frac{1}{\beta}$ and $\alpha=0$, $\beta=1$.
In order to prove the long time existence of the smooth solution to the flow \eqref{NSF}, we need to prove the a priori estimates ($C^0$ estimates, $C^1$ estimates and $C^2$ estimates) by the Evans-Krylov's regularity theory for parabolic equations.  The key step is to get the $C^0$ estimates and the uniform upper bound of $\eta(t)$ in our argument. We conclude the flow \ref{NSF} exists for all times $t>0$ and $u(\cdot, t)$ remains positive, smooth and uniformly convex. By the monotonicity of $\mathcal{J}(u(\cdot,t))$, there is a sequence of $t_i\to\infty$ such that $u(\cdot,t_i)\to u_\infty(\cdot)$ which solves \eqref{elliptic eq}, where $c=\lim_{t_i\to\infty}\eta(t_i)$ is a positive constant.

In this paper, we will prove the following
\begin{theorem}\label{main1}
Let $\mathcal{M}_0$ be a smooth, closed, uniformly convex, and origin-symmetric hypersurface in $R^{n+1}$, $n\ge2$, enclosing the origin. For the cases
$1-\beta<\alpha<1+n\beta$ and $\alpha=0$, $\beta=1$, respectively,
the flow \eqref{SF1} has a unique smooth and uniformly convex solution $\mathcal{M}_t$ provided that $f$ is a smooth positive even function on $S^n$. After normalisation, the rescaled hypersurfaces $\widetilde{\mathcal{M}_t}$ converge smoothly to a smooth solution of \eqref{elliptic eq}, which is a minimiser of the functional \eqref{functional}.
\end{theorem}

\begin{theorem}\label{main2}
Let $\mathcal{M}_0$ be a smooth, closed, uniformly convex, and origin-symmetric hypersurface in $R^{n+1}$, $n\ge2$, enclosing the origin. When $1-n\beta-2\beta<\alpha<1-\beta$,
suppose $f$ is a smooth positive even function on $S^n$, then the flow \eqref{SF1} has a unique smooth and uniformly convex solution $\mathcal{M}_t$. After normalisation, the rescaled hypersurfaces $\widetilde{\mathcal{M}_t}$ converge smoothly to a smooth solution of \eqref{elliptic eq}, which is a maximiser of the functional \eqref{functional}.
\end{theorem}

\begin{theorem}\label{main3}
Let $\mathcal{M}_0$ be a smooth, closed and uniformly convex hypersurface in $R^{n+1}$, $n\ge2$, enclosing the origin. Suppose $\alpha\ge1+n\beta$, $\beta>0$,
Then for any smooth positive function $f$ on $S^n$, the flow \eqref{SF1} has a unique smooth and uniformly convex solution $\mathcal{M}_t$. After normalisation, the rescaled hypersurfaces $\widetilde{M_t}$ converge smoothly to a unique smooth solution of \eqref{elliptic eq}, which is a minimiser of the functional \eqref{functional}.
\end{theorem}

\begin{remark}
In this paper, we focus on the convergence of the normalized flow \eqref{NSF} by discussing the relationship between $\alpha$ and $\beta$.  When $1<\alpha<1+n\beta$, we prove the uniqueness of the solution to the elliptic equation $fu^\alpha\sigma_n^{-\beta}=c$ in Section 4 Proposition \ref{alpha}. Hence the rescaled hypersurfaces $\widetilde{M_t}$ converge smoothly to a unique smooth solution of \eqref{elliptic eq} for $\alpha>1$.
\end{remark}

By Theorems \ref{main1}-\ref{main3}, we obtain the following result for $L_p$ Minkowski problem.
\begin{cor}
Let $M$ be a smooth, closed and uniformly convex hypersurface in $R^{n+1}$, $n\ge2$, enclosing the origin.
\begin{itemize}
\item[(i)] When $-n-1<p<n+1$, suppose $M$ is origin-symmetric and $\phi$ is a smooth positive even function on $S^n$, then the $L_p$ Minkowski problem $u^{1-p}\sigma_n([\nabla^2 u+uI])=\phi$ has an origin-symmetric smooth solution;
\item [(ii)] When $p\ge1+n$ and $\phi$ is a smooth positive function on $S^n$, then the $L_p$ Minkowski problem  $u^{1-p}\sigma_n([\nabla^2 u+uI])=\phi$ has a unique smooth solution. The uniqueness for $p=n+1$ is up to a dilation.
\end{itemize}
\end{cor}

This paper is organised as follows. In Section 2, we recall some properties of convex hypersurfaces. We give the uniform upper bound on $\eta(t)$ to ensure the normalised flow \eqref{NSF} being well-defined, and show that the functional \eqref{functional} is strictly monotone along the flow \eqref{NSF} unless $u$ satisfies the elliptic equation \eqref{elliptic eq}. In Section 3, we establish the a priori estimates, which implies the uniqueness and the long time existence of the normalised flow \eqref{NSF}. In Section 4, we prove Theorems \ref{main1}-\ref{main3}. We also give the proof of the uniqueness of the elliptic equation \eqref{elliptic eq} for the case $1<\alpha<1+n\beta$ in Proposition \ref{alpha}.

\section{Preliminary}
We recall some basic notations at first.
Let $\mathcal{M}$ be a smooth, closed, uniformly convex hypersurface in $R^{n+1}$, enclosing the origin.  Assume that $M$ is parametrized by the inverse Gauss map $X:S^n\to \mathcal{M}\subset R^{n+1}$ and encloses origin.
The radial function $r$ is defined by
\[X=r\xi,\]where $\xi=\frac{X}{|X|}$ is the unit radial vector.
The support function $u:S^n\to R^1$ of $\mathcal{M}$ is defined by
\[u(x)=\sup_{y\in \mathcal{M}}\langle x,y\rangle.\]
The supermum is attained at a point $y=X(x)$, $x$ is the outer normal of $\mathcal{M}$ at $y$. Hence
\[
u(x)=\langle x,X(x)\rangle.
\]
Let $e_1, \cdots, e_n$ be a smooth local orthonormal frame field on $S^n$, and ${\nabla}$ the covariant derivative on $S^n$. Denote by $g_{ij}$, $g^{ij}$, $h_{ij}$ the metric, the inverse of the metric and the second fundamental form of $\mathcal{M}$, respectively. Then the second fundamental form of $\mathcal{M}$ is given by (see e.g.\cite{Urb91})
$$h_{ij}=\nabla_i\nabla_ju+u\delta_{ij}.$$
By the Gauss-Weingarten formula
$$\nabla_ix=h_{jk}g^{kl}\nabla_lX,$$
we get
$$\delta_{ij}=\langle\nabla_ix,\nabla_jx\rangle=h_{ik}g^{kl}h_{jm}g^{ms}\langle\nabla_lX,\nabla_sX\rangle=g^{kl}h_{ik}h_{jl}.$$
Since $M$ is uniformly convex, $h_{ij}$ is invertible. Hence the principal curvature radii are the eigenvalues of the matrix
 \[
 b_{ij}=h^{ik}g_{jk}=h_{ij}=\nabla_{ij}u+u\delta_{ij},
 \]
By a simple calculation (see \cite{LSW16}), we know
\begin{eqnarray}\label{gij}
g_{ij}=r^2\delta_{ij}+r_ir_j,
\end{eqnarray}
\begin{eqnarray}
x=\frac{r\xi-{\nabla}r}{\sqrt{r^2+|{\nabla}r|^2}},
\end{eqnarray}
\begin{eqnarray}\label{hbarij}
{\hbar}_{ij}=\frac{-rr_{ij}+2r_ir_j+r^2\delta_{ij}}{\sqrt{r^2+|{\nabla}r|^2}},
\end{eqnarray}
\begin{eqnarray}\label{nablau}
r=\sqrt {u^2+|\nabla u|^2},
\end{eqnarray}
\begin{eqnarray}\label{nablar}
u=\frac{r^2}{\sqrt{r^2+|{\nabla}r|^2}}.
\end{eqnarray}
Let $\Omega$ be a convex body enclosing the origin, $\partial\Omega=\mathcal{M}$. The dual body of $\Omega$ with respect to the origin, denoted by $\Omega^*$, is defined as
\begin{eqnarray}
\Omega^*=\{y\in R^{n+1}| x\cdot y\le1, \forall x\in\Omega\}.
\end{eqnarray}
Its support function $u^*(\xi,t)=\frac{1}{r(\xi,t)}$, and its radial function $r^*(x,t)=\frac{1}{u(x,t)}$ (see \cite{HLYZ16} for details).

Next we introduce some basic concepts about the Minkowski mixed volume $V_{n+1}(u^1,u^2,\ldots,u^{n+1})$, where $u^1,u^2,\ldots,u^{n+1}$ are the support functions of some convex bodies $\Omega_1,\Omega_2,\ldots,\Omega_{n+1}$ respectively.
Let $\sigma_k(A)$, $1\le k\le n$, be the $k$-th elementary symmetric function defined on the set $\mathcal{M}_n$ of $n\times n$ matrices and $\sigma_k(A_1,\ldots,A_k)$ be the complete polarization of $\sigma_k$ for $A_i\in\mathcal{M}_n$, $i=1,\ldots,k$,
i.e.
\[\sigma_k(A_1,\ldots,A_k)=\frac{1}{k!}\sum_{{i_1,\ldots,i_k=1}
; {j_1,\ldots,j_k=1}}^n\delta_{j_1,\ldots,j_k}^{i_1,\ldots,i_k}{({A}_{1}})_{i_1j_1}\cdots ({A_{k}})_{i_kj_k}.\]
Let $\Gamma_k$ be Garding's cone
\[\Gamma_k=\{A\in\mathcal{M}_n:\sigma_i(A)>0, i=1,\ldots,k\}.\]
For a function $u\in C^2(S^n)$, we denote by $W_u$ the matrix
\[W_u:=\nabla^2u+uI.\]
In the case $W_u$ is positive definite, the eigenvalue of $W_u$ is the principal radii of a strictly convex hypersurface with support function $u$.
Let $u^i\in C^2(S^n)$, $i=1,\ldots,n+1$. Set
\[
V_{n+1}(u^1,u^2,\ldots,u^{n+1}):=\int_{S^n}u^1\sigma_n[W_{u^2},\ldots,W_{u^{n+1}}]dx,
\]
\[
V_{k+1}(u^1,u^2,\ldots,u^{k+1}):=V_{n+1}(u^1,u^2,\ldots,u^{k+1},1,\ldots,1).
\]
Here, we state the well-known Alexandrov-Fenchel inequality.

\begin{lem}\label{AFI}(\cite{H94})
Let $u^i\in C^2(S^n)$, $i=1,2,\ldots,k$ be such that $u^i>0$ and $W_{u^i}\in\Gamma_k$ for $i=1,2,\ldots,k$. Then for any $v\in C^2(S^n)$, the Alexandrov-Fenchel inequality holds:
\[V_{k+1}(v,u^1,\ldots,u^k)^2\ge V_{k+1}(v,v,u^2,\ldots,u^k)V_{k+1}(u^1,u^1,u^2,\ldots,u^k),\]
the equality holds if and only if $v=au^1+\sum_{l=1}^{n+1}a_lx_l$ for some constants $a,a_1,\ldots,a_{n+1}$.
\end{lem}
We consider the flow \eqref{NSF}. We set
\begin{center}
$\rho=fu^{\alpha-1}\sigma_n^{-\beta}$,
\qquad$\sigma[f]=\sigma_n[W_f,W_u,\ldots,W_u]$.
\end{center}
Then the flow \eqref{NSF} can be written as $\frac{\partial u}{\partial t}=-\rho u+u\eta(t)$, and
 $\mathcal{Z}_p(u(\cdot,t))=\int_{S^n}u\sigma_n(fu^{\alpha-1}\sigma_n^{-\beta})^pdx=\int_{S^n}u\sigma_n\rho^pdx$, where $p\in R^1$, $\eta(t)=\frac{\mathcal{Z}_1}{|S^n|}$,
and $\mathcal{Z}_0(u(\cdot,t))=\int_{S^n} u\sigma_n dx=|S^n|$. By a similar calculation in \cite{And97}, we have
\begin{eqnarray}
&&\frac{d}{dt}\mathcal{Z}_p(u(\cdot,t))\nonumber\\
&=&\int_{S^n}(-\rho u+u\frac{\mathcal{Z}_1}{|S^n|})\sigma_n\rho^{p}dx+\int_{S^n}nu\sigma[-\rho u+u\frac{\mathcal{Z}_1}{|S^n|}]\rho^{p}dx
\nonumber\\
&&+\int_{S^n}p\rho^{p-1}u\sigma_n\Big((\alpha-1)fu^{\alpha-2}(-\rho u+u\frac{\mathcal{Z}_1}{|S^n|})\sigma_n^{-\beta}
-n\beta fu^{\alpha-1}\sigma_n^{-\beta-1}\sigma[-\rho u+u\frac{\mathcal{Z}_1}{|S^n|}]\Big)dx
\nonumber\\
&=&-\mathcal{Z}_{1+p}+\frac{\mathcal{Z}_1\mathcal{Z}_{p}}{|S^n|}-n\int_{S^n}u\sigma[\rho u]\rho^{p}dx+n\frac{\mathcal{Z}_1\mathcal{Z}_{p}}{|S^n|}
\nonumber\\
&&-p(\alpha-1)\mathcal{Z}_{1+p}+p(\alpha-1)\frac{\mathcal{Z}_1\mathcal{Z}_{p}}{|S^n|}+n\beta p\int_{S^n}u\sigma[\rho u]\rho^{p}dx-n\beta p\frac{\mathcal{Z}_1\mathcal{Z}_{p}}{|S^n|}
\nonumber\\
&=&-\big(1+(\alpha-1)p\big)\Big(\mathcal{Z}_{1+p}-\frac{\mathcal{Z}_1\mathcal{Z}_p}{|S^n|}\Big) -n(1-\beta p)\Big(\int_{S^n}u\rho^p\sigma[\rho u]dx-\frac{\mathcal{Z}_1\mathcal{Z}_p}{|S^n|}\Big).
\nonumber\end{eqnarray}
Since $h_{ij}$ satisfies Codazzi equations,
 we have $\sum_i\nabla_i\sigma^{ij}=0$ (\cite{And94(1)}, \cite{And97}), and
\begin{eqnarray*}
\int_{S^n}u\rho^p\sigma_n[\rho u,u,\ldots,u]dx&=&\int_{S^n}u\rho^p\sigma[\rho u]d\mu\\
&=&\int_{S^n}u\rho^p\sigma^{ij}\big(\nabla_i\nabla_j(u\rho)+\delta_{ij} u\rho\big)dx
\\&=&\int_{S^n}u\rho^p\sigma^{ij}(h_{ij}\rho+2\nabla_iu\nabla_j\rho+u\nabla_i\nabla_j\rho)dx
\\&=&\mathcal{Z}_{1+p}-p\int_{S^n}u^2\rho^{p-1}\sigma^{ij}\nabla_i\rho \nabla_j\rho dx
\\&=& \mathcal{Z}_{1+p}-\frac{4p}{(1+p)^2}\int_{S^n}u^2\sigma^{ij}\nabla_i(\rho^{\frac{1+p}{2}})\nabla_j(\rho^{\frac{1+p}{2}})dx.
\end{eqnarray*}
By the Alexandrov-Fenchel inequality in Lemma \ref{AFI}, we have
\begin{eqnarray}\label{AF}
\Big(\int_{S^n}u\psi\sigma_n[u,u,\ldots,u]dx\Big)^2&\ge&\int_{S^n}u\sigma_n[u,u,\ldots,u]dx\int_{S^n}u\psi\sigma_n[u\psi ,u,\ldots,u]dx
\nonumber\\&=&|S^n|\Big(\int_{S^n}u\sigma_n\psi^2dx-\int_{S^n}u^2\sigma^{ij}\nabla_i\psi\nabla_j\psi dx\Big),
\end{eqnarray}
Set $\psi=\rho^{\frac{1+p}{2}}$ in the Alexandrov-Fenchel inequality \eqref{AFI}, we obtain
\[\int_{S^n}u^2\sigma^{ij}\nabla_i\rho^{\frac{1+p}{2}}\nabla_j\rho^{\frac{1+p}{2}} d\mu-\mathcal{Z}_{1+p}+\frac{\mathcal{Z}^2_{\frac{1+p}{2}}}{|S^n|}\ge0.\]
Thus
\begin{eqnarray*}
\frac{d}{dt}\mathcal{Z}_p(u(\cdot,t))&=&-[1+(\alpha-1)p+n(1-p\beta)]\Big(\mathcal{Z}_{p+1}- \frac{\mathcal{Z}_1\mathcal{Z}_p}{|S^n|}\Big)
+\frac{4p n(1-p\beta)}{(1+p)^2}\Big(\mathcal{Z}_{1+p}-\frac{\mathcal{Z}^2_{\frac{1+p}{2}}}{|S^n|}\Big)
\\&{}&+\frac{4p n(1-p\beta)}{(1+p)^2}\Big(\int_{S^n}u^2\sigma^{ij}\nabla_i\rho^{\frac{1+p}{2}}\nabla_j\rho^{\frac{1+p}{2}} dx-\mathcal{Z}_{1+p}+\frac{\mathcal{Z}^2_{\frac{1+p}{2}}}{|S^n|}\Big).
\end{eqnarray*}

\begin{lem}\label{upboundeta}
$\eta(t)$ has a uniform upper bound for the cases $\alpha\ge0$, $\beta>1$; $\alpha>1-n\beta-2\beta$, $\alpha\ne1-\beta$, $0<\beta\le1$ and $\alpha=0$, $\beta=1$, respectively.
\end{lem}
\begin{proof}
Let $p=1$, we have
\begin{eqnarray*}
\frac{d}{dt}\mathcal{Z}_1
&=&-\alpha\big(\mathcal{Z}_2-\frac{\mathcal{Z}^2_1}{|S^n|}\big)+n(1-\beta)\big(\int_{S^n}u^2\sigma^{ij}\nabla_i\rho\nabla_j\rho dx-\mathcal{Z}_{2}+\frac{\mathcal{Z}^2_{1}}{|S^n|}\big)
\end{eqnarray*}
where the H\"{o}lder inequality shows that $\mathcal{Z}_{2}\ge\frac{\mathcal{Z}^2_{1}}{|S^n|}$.

Case (i): $\alpha\ge0$, $\beta>1$, we obtain $\frac{d}{dt}\eta=\frac{\frac{d}{dt}\mathcal{Z}_1}{|S^n|}\le0$,  then $\eta(t)\le C$, where $C$ depending on the initial hypersuface.

Case (ii): $\alpha>1-n\beta-2\beta$, $\alpha\ne1-\beta$, $0<\beta\le1$,  we have
$\mathcal{Z}_{1}(u)\le\Big(\mathcal{Z}_{0}(u)\Big)^{1-\beta}\Big(\mathcal{Z}_{\frac{1}{\beta}}(u)\Big)^\beta$ by the H\"{o}lder inequality. Hence we only need to prove that $\mathcal{Z}_{\frac{1}{\beta}}(u)\le C$, for some positive constant $C$.  Let $p=\frac{1}{\beta}$, we have
\begin{eqnarray*}
\frac{d}{dt}\mathcal{Z}_{\frac{1}{\beta}}(u)&=&
\frac{1-\alpha-\beta}{\beta}\Big(\mathcal{Z}_{1+\frac{1}{\beta}}-\frac{\mathcal{Z}_1\mathcal{Z}_{\frac{1}{\beta}}}{|S^n|}\Big).
\end{eqnarray*}
For $\alpha>1-\beta$, we have $\frac{d}{dt}\mathcal{Z}_{\frac{1}{\beta}}(u)\le0$ since $\mathcal{Z}_{1+\frac{1}{\beta}}\ge\frac{\mathcal{Z}_1\mathcal{Z}_{\frac{1}{\beta}}}{|S^n|}$ by the H\"{o}lder inequality. Then $\mathcal{Z}_{\frac{1}{\beta}}(u)\le C$, and $\eta(t)\le C$, where $C$ depends on the initial hypersuface.
For $1-n\beta-2\beta<\alpha<1-\beta$,
$\frac{d}{dt}\mathcal{Z}_{\frac{1}{\beta}}(u)\ge0$
since $\mathcal{Z}_{1+\frac{1}{\beta}}\ge\frac{\mathcal{Z}_1\mathcal{Z}_{\frac{1}{\beta}}}{|S^n|}$ by the H\"{o}lder inequality. Hence  $$\mathcal{Z}_{\frac{1}{\beta}}(u_0)\le\mathcal{Z}_{\frac{1}{\beta}}(u)=\int_{S^n}f^{\frac{1}{\beta}}u^{\frac{\alpha-1+\beta}{\beta}} dx\le(\mathop{\max}_{S^n}f)^\frac{1}{\beta}\int_{S^n}u^{\frac{\alpha-1+\beta}{\beta}} dx,$$
that is
\begin{equation}\label{intbound}
c=\mathcal{Z}_{\frac{1}{\beta}}(u_0)(\mathop{\max}_{S^n}f)^{-\frac{1}{\beta}}\le\int_{S^n}u^{\frac{\alpha-1+\beta}{\beta}} dx.
\end{equation}
In this part, we shall use the Blaschke-Santal\'{o} inequality
\[\mathrm{Vol}(\Omega)\mathrm{Vol}(\Omega^*)\le\mathrm{Vol}(B_1)^2,\]
where $\Omega$ is the convex body enclosing the origin, $\Omega^*$ is the polar body of $\Omega$, $\mathrm{Vol}(\Omega)=\int_{S^n}r^{n+1}d\xi$, $\mathrm{Vol}(\Omega^*)=\int_{S^n}{r^*}^{n+1}dx=\int_{S^n}(\frac{1}{u})^{n+1}dx$, the equality holds if and only if $\Omega$ is a ellipsoid.

Set $q=\frac{\alpha-1+\beta}{\beta}$, $-n-1<q<0$, we refer to the result of Chou-Wang\cite{ChWang06}: If origin-symmetric convex body $\Omega$ satisfies $c\le\int_{S^n}u^{q} dx$, $q<0$, $\mathrm{Vol(\Omega)}=\int_{S^n}u\sigma_n dx=|S^n|$, then the diameter of convex body $\Omega$ enclosed by $\mathcal{M}$, $d(\Omega)\le C$, for some positive $C$, where $d(\Omega)=2\mathop{\max}_{S^n}u$ for the origin-symmetric convex body $\Omega$. We give the same argument as follows. Suppose there is a sequence origin-symmetric convex body $\Omega_{t_j}$ satisfying \eqref{intbound}, but the diameter of $\Omega_{t_j}$, $d_j\to \infty$ as $t_j\to T$. Let $\frac{E_{t_j}}{n+1}$ be the origin-symmetric John ellipsoid associated with $\Omega_{t_j}$, as is well known, see \cite{Sch14}, $\frac{E_{t_j}}{n+1}\subset \Omega_{t_j}\subset E_{t_j}$, $\frac{u_{E_j}}{n+1}<u_j<u_{E_j}$. we set $S^n=S_1\cup S_2\cup S_3$, where
\[S_1=S^n\cap\{u_{E_j}<\delta\},\qquad S_2=S^n\cap\{\delta\le u_{E_j}<\frac{1}{\delta}\}, \qquad S_3=S^n\cap\{u_{E_j}\ge\frac{1}{\delta}\}.\]
where $\delta\in(0,\frac{1}{4})$ is a fixed constant. Then
\[c\le\int_{S^n}u_{j}^{q}dx<\int_{S^n}(\frac{u_{E_j}}{n+1})^{q}dx.\]
Suppose $u_j$ attains the maximum at $x_0$, where $x_0\in S^n$, that is, $u_j(x_0)=\mathop{\max}_{S^n}u_j$, and $\mathop{\max}_{S^n}u_j=\mathop{\max}_{S^n}r_j$ by \eqref{nablau}. Since $u_j(y)\ge\frac{1}{2}d_j|x_0\cdot y|$ for any $y\in S^n$, we obtain $|S_1|$, $|S_2|\to 0$ as $d_j\to\infty$.

As $d_j\to\infty$, for any fixed $\delta$, we have
\[\int_{S_1}(\frac{u_{E_j}}{n+1})^{q}dx\le(\frac{1}{n+1})^q\Big(\int_{S^n} \frac{1}{u^{n+1}_{E_j}}\Big)^\frac{-q}{n+1}|S_1|^{\frac{q+n+1}{n+1}}\le C_1|S_1|^{\frac{q+n+1}{n+1}}\to 0,\]by the Blaschke-Santal\'{o} inequality.
Noting $|S_2|\to0$ as $d_j\to\infty$, and
\[\int_{S_3}(\frac{u_{E_j}}{n+1})^qdx\le\int_{S_3}\big(\frac{1}{\big(n+1)\delta}\big)^qdx=\big(\frac{1}{(n+1)\delta}\big)^q|S_3|\le C_2\delta^{-q}.\]
Hence, we have \[c\le\circ(1)+C_3\delta^{-q}.\]
for any $\delta\in(0,\frac{1}{4})$. Let $\delta\to0$, we reach a contradiction. It implies ${\mathop{\max}_{S^n}}u(\cdot,t)\le C$, for some positive constant $C$.

Next we derive the lower bound for $u(\cdot,t)$. It is well known that \[\int_{S^n}u(x)\sigma_ndx=\int_{S^n}r^{n+1}(\xi)d\xi=\mathrm{Vol}(\Omega_t),\]
where $\Omega_t$ denotes the convex body enclosed by $\mathcal{M}_t$. By \eqref{nablau}, it is easy to see $r_{max}(t)=u_{max}(t)$, $r_{min}(t)=u_{min}(t)$.
We may assume that $r_{\max}(t)=\mathop{\max}_{S^n}r(e_1,t)$ and $r_{min}(t)=r(e_{n+1},t)$ by rotating the coordinates. Since $\Omega_t$ is origin-symmetric, we find that $\Omega_t$ is contained in a cube
\[
Q_t=\{z\in R^{n+1}: -r_{max}(t)\le z_i\le r_{max}(t)\, {\rm{for}}\,    1\le i\le n, -r_{min}(t)\le z_{n+1}\le r_{min}(t) \}.\]
Therefore
\[|S^n|=\mathrm{Vol}(\Omega_t)\le 2^{n+1}r^{n}_{max}(t)r_{min}(t)\]
Using $r_{max}(t)\le C$, we get $r_{min}(t)\ge\frac{1}{C}$ for some positive constant $C$, then\\ $\mathcal{Z}_{\frac{1}{\beta}}=\int_{S^n}f^{\frac{1}{\beta}}u^{\frac{\alpha-1+\beta}{\beta}} dx\le(\mathop{\max}_{S^n}f)^{\frac{1}{\beta}}{u^{\frac{\alpha-1+\beta}{\beta}} _{min}} |S^n|\le C$. Hence $\mathcal{Z}_1\le C$, for some positive constant $C$.

Case (iii): $\alpha=0$, $\beta=1$, we obtain $\eta(t)=\frac{\int_{S^n}f dx}{|S^n|}=c$, where $c$ is a positive constant.
\end{proof}

In Case (ii) of the proof, we have obtained the  $C^0$ estimates of the solutions to the equation \eqref{NSF}:  $\frac{1}{C}\le u\le C$ for the case $1-n\beta-2\beta<\alpha<1-\beta$ for some positive constant $C$.

When $\alpha>1+n\beta$, $\beta>0$, we also need the uniform lower bound on $\eta(t)$ to obtain the priori estimate in the next section.
\begin{lem}\label{boundeta}
Suppose $\alpha>1+n\beta$, $\beta>0$, $\eta(t)$ is uniformly bounded.
\end{lem}
\begin{proof}
Since $\alpha>1+n\beta$, $\beta>0$, we set $\theta\le\frac{1+n}{1+n\beta-\alpha}<0$, $\alpha>1+n\beta$,  we have
\begin{eqnarray*}
\frac{d}{dt}\mathcal{Z}_\theta(u)&=&-[1+(\alpha-1)\theta+n(1-\theta\beta)]\Big(\mathcal{Z}_{\theta+1}- \frac{\mathcal{Z}_1\mathcal{Z}_\theta}{|S^n|}\Big)
+\frac{4\theta n(1-\theta\beta)}{(1+\theta)^2}\Big(\mathcal{Z}_{1+\theta}-\frac{\mathcal{Z}^2_{\frac{1+\theta}{2}}}{|S^n|}\Big)
\\&{}&+\frac{4\theta n(1-\theta\beta)}{(1+\theta)^2}\Big(\int_{S^n}u^2\sigma^{ij}\nabla_i\rho^{\frac{1+\theta}{2}}\nabla_j\rho^{\frac{1+\theta}{2}} dx-\mathcal{Z}_{1+\theta}+\frac{\mathcal{Z}^2_{\frac{1+\theta}{2}}}{|S^n|}\Big)
\\&\le&0
\end{eqnarray*}
since $\theta\le\frac{1+n}{1+n\beta-\alpha}<0$, and by the H\"{o}lder inequality, we get
$\mathcal{Z}_{1+\theta}\le\frac{\mathcal{Z}_1\mathcal{Z}_{\theta}}{|S^n|}$ and
$\mathcal{Z}_{1+\theta}\ge\frac{\mathcal{Z}^2_\frac{1+\theta}{2}}{|S^n|}$.
Hence, $\mathcal{Z}_\theta(u)\le\mathcal{Z}_\theta(u_0)$.
By the H\"{o}lder inequality again, we have
\[
|S^n|=\int_{S^n}u\sigma_ndx\le\big(\int_{S^n}fu^\alpha\sigma_n^{1-\beta}dx\big)^{\frac{-\theta}{1-\theta}}\big(\int_{S^n}u\sigma_n(fu^{\alpha-1}\sigma_n^{-\beta})^\theta dx\big)^{\frac{1}{1-\theta}}=\mathcal{Z}^{\frac{-\theta}{1-\theta}}_1\mathcal{Z}^{\frac{1}{1-\theta}}_\theta.
\]
It is easy to see, $\mathcal{Z}_{1}\ge C$, by case(i) and case(ii) in Lemma \ref{upboundeta}, we get the uniform bound on $\eta(t)$ for $\alpha>1+n\beta$, $\beta>0$.
\end{proof}

\begin{lem}\label{monotone12}
The functional \eqref{functional} is non-increasing along the normalised flow \eqref{NSF} for the case $\alpha>1-\beta$, $\beta>0$,
  and the equality holds if and only if $\mathcal{M}_t$ satisfies the elliptic equation \eqref{elliptic eq}.
\end{lem}
\begin{proof}
From the above calculation process, when $p=\frac{1}{\beta}$, we obtain along the normalised flow \eqref{NSF}
\begin{eqnarray*}
\frac{d}{dt}\mathcal{J}(u)=\frac{d}{dt}\mathcal{Z}_{\frac{1}{\beta}}(u)&=&
\frac{1-\alpha-\beta}{\beta}\Big(\mathcal{Z}_{1+\frac{1}{\beta}}-\frac{\mathcal{Z}_1\mathcal{Z}_{\frac{1}{\beta}}}{|S^n|}\Big)
\\&\le&0.
\end{eqnarray*}
The last inequality holds from the H\"{o}lder inequality,
and the equality holds if and only if $fu^{\alpha-1}\sigma_n^{-\beta}=c(t)$ for some function $c(t)$. Indeed, by \eqref{def eta t}, if $fu^{\alpha-1}\sigma_n^{-\beta}=c(t)$ occurs, then
\begin{eqnarray*}
\eta(t)=\frac{\int_{S^n}fu^\alpha\sigma_n^{1-\beta}dx}{|S^n|}=\frac{\int_{S^n}u\sigma_nc(t)dx}{|S^n|}=c(t).
\end{eqnarray*}
\end{proof}
\begin{lem}\label{monotone3}
The functional \eqref{functional} is non-decreasing along the normalised flow \eqref{NSF} for the case $1-n\beta-2\beta<\alpha<1-\beta$,
 and the equality holds if and only if $\mathcal{M}_t$ satisfies the elliptic equation \eqref{elliptic eq}.
\end{lem}
\begin{proof}
From the above calculation, when $p=\frac{1}{\beta}$, we obtain along the normalised flow \eqref{NSF}
\begin{eqnarray*}
\frac{d}{dt}\mathcal{J}(u)=\frac{d}{dt}\mathcal{Z}_{\frac{1}{\beta}}(u)&=&
\frac{1-\alpha-\beta}{\beta}\Big(\mathcal{Z}_{1+\frac{1}{\beta}}-\frac{\mathcal{Z}_1\mathcal{Z}_{\frac{1}{\beta}}}{|S^n|}\Big)
\\&\ge&0.
\end{eqnarray*}
The last inequality holds from the H\"{o}lder inequality,
and the equality holds if and only if $fu^{\alpha-1}\sigma_n^{-\beta}=c(t)$ for some function $c(t)$. In the same way as in the proof of Lemma \ref{monotone12}, we can show  $\eta(t)=c(t)$.
\end{proof}
For $\alpha=0$, $\beta=1$, it it easy to see, $\eta(t)=\frac{\int_{S^n}fdx}{|S^n|}=c$, where $c$ is a positive constant.
\begin{lem}\label{monotone4}
The functional \eqref{functional} is non-increasing along the normalised flow \eqref{NSF} for $\alpha=0$, $\beta=1$, and the equality holds if and only if $\mathcal{M}_t$ satisfies the elliptic equation \eqref{elliptic eq}.
\end{lem}
\begin{proof}
\begin{eqnarray*}
\frac{d}{dt}\mathcal{J}(u)&=&\frac{\int_{S^n}fu^{-1}u_tdx-\frac{\int_{S^n}f dx}{\int_{S^n}u\sigma_n dx}{\int_{S^n}u_t\sigma_n dx}}{\int_{S^n}f dx}
\\&=&\frac{\int_{S^n}u_t\big(fu^{-1}-\eta\sigma_n\big)dx}{\int_{S^n}f dx}
\\&=&\frac{-\int_{S^n}u^{-1}\sigma_n\big(f\sigma_n^{-1}-\eta u\big)^2 dx}{\int_{S^n}f dx}
\\&\le&0
\end{eqnarray*}
The equality holds if and only if $fu^{-1}\sigma_n^{-1}=c$ where $c=\frac{\int_{S^n}fdx}{|S^n|}$ is a positive constant.
\end{proof}

\section{A priori estimates}
We firstly show the uniformly lower and upper bound of the solution to \eqref{NSF}.
\begin{lem}\label{C0estimate1}
Let $\mathcal{M}_t$, $t\in [0,T)$, be an origin-symmetric solution to  \eqref{NSF}. For the following cases: $1-n\beta-2\beta<\alpha<1+n\beta$, $\alpha\ne1-\beta$ and $\alpha=0$, $\beta=1$,
  there is a positive constant $C$ depending only on $\alpha$, $\beta$, $f$ and initial hypersurface, such that
 \[
 \frac{1}{C}\le u(\cdot, t)\le C.
 \]
\end{lem}
\begin{proof}
Let $r_{\min}(t)={\mathop{\min}_{S^n}}r(\cdot,t)$ and $r_{\max}(t)=\mathop{\max}_{S^n}r(\cdot,t)$. We may assume that $r_{\max}(t)=\mathop{\max}_{S^n}r(e_1,t)$ by rotating the coordinates. Since $\widetilde{\mathcal{M}}_t$ is origin-symmetric, the points $\pm r_{max}(t)e_1\in \widetilde{\mathcal{M}}_t$. Hence
\[u(x,t)=\sup\{p\cdot x: p\in \widetilde{\mathcal{M}}_t\}\ge r_{\max}|x\cdot e_1|, \forall{x}\in S^n.\]

For the case $1-\beta<\alpha<1+n\beta$,
we obtain
\begin{eqnarray*}
\mathcal{J}(u)=\int_{S^n}f^{\frac{1}{\beta}}u^{\frac{\alpha-1+\beta}{\beta}}dx
\ge r^{\frac{\alpha-1+\beta}{\beta}}_{max}(t) \int_{S^n}f^{\frac{1}{\beta}}|x\cdot e_1|^{\frac{\alpha-1+\beta}{\beta}}dx \ge C_0(\min_{S^n}f)^{\frac{1}{\beta}}r^{\frac{\alpha-1+\beta}{\beta}}_{max}(t),
\end{eqnarray*}
where $\alpha-1+\beta>0$.
By Lemma \ref{monotone12}, $\frac{d}{dt}\mathcal{J}(u)\le0$, we conclude
\[\mathcal{J}(u_0)\ge\mathcal{J}(u(t))\ge C_0(\min_{S^n}f)^{\frac{1}{\beta}}r^{\frac{\alpha-1+\beta}{\beta}}_{max}(t).\]
This implies $r_{max}\le C$ for some positive constant depending on $\alpha$, $\beta$, $f$ and initial hypersurface.

For the case $1-n\beta-2\beta<\alpha<1-\beta$,
 the uniform bounds of $u(\cdot,t)$ is obtained from the proof case (ii) in Lemma \ref{upboundeta}.

Now we consider the case $\alpha=0$, $\beta=1$. For $\mathcal{J}(u)=\frac{\int_{S^n}f \log u dx}{\int_{S^n}f dx}-\frac{1}{n+1}\log\int_{S^n}u\sigma_n dx$, we have proved $\frac{d}{dt}\mathcal{J}(u)\le0$. Since
\begin{eqnarray*}
\int_{S^n}f(x)\log u(x,t) dx&\ge&\big(\int_{S^n}f(x)dx\big)\log r_{\max}(t)+\int_{S^n}f(x)\log |x\cdot e_1| dx
\\&\ge&\big(\int_{S^n}f(x)dx\big)\log r_{\max}(t)-C\mathop{\max}_{S^n}f,
\end{eqnarray*}
we have $\mathcal{J}(u_0)\ge\mathcal{J}(u(t))\ge\log r_{\max}(t)-C\frac{\mathop{\max}_{S^n}f}{|S^n|\mathop{\min}_{S^n}f}-\frac{1}{n+1}\log |S^n|$, which implies $r_{\max}\le C$.
Since $u_{\max}(t)=r_{\max}(t)$, we therefore get the uniformly upper bound of $u(\cdot,t)$.
For origin-symmetric convex body $\Omega_t$, by rotating the coordinates and constructing the cube $Q_t$ just as the same way of Case (ii) in the proof of Lemma \ref{upboundeta}, we have
 \[
 |S^n|=\mathrm{Vol}(\Omega_t)\le 2^{n+1}r^{n}_{\max}(t)r_{\min}(t).
 \]
 Therefore we get the uniform lower bound of $u(\cdot,t)$ since $u_{\min}(t)=r_{\min}(t)$. Hence we complete the proof.
\end{proof}

\begin{lem}\label{C0estimate2}
Let $\mathcal{M}_t$, $t\in [0,T)$, be a solution to  \eqref{NSF}. If $\alpha\ge n\beta+1$ and $\beta>0$, there is a positive constant $C$ depending only on $\alpha$, $\beta$ and the initial hypersurface such that
                          \[\frac{1}{C}\le u(., t)\le C.\]
\end{lem}
\begin{proof}
For the case $\alpha>n\beta+1$,  let $u_{\min}(t)=\min_{x\in S^n}u(.,t)$, we have
\[\frac{du_{\min}}{dt}\ge- u_{\min}(fu_{\min}^{\alpha-n\beta-1}-\eta).\]
Hence,
$u_{\min}\ge \min\{(\frac{\min\eta}{\mathop{\max}_{S^n}f})^\frac{1}{\alpha-n\beta-1}, {u_{\min}(0)}\}.$

Similarly, we have $u_{\max}\le \max \{(\frac{\max\eta}{\mathop{\min}_{S^n}f})^\frac{1}{\alpha-n\beta-1}, {u_{\max}(0)}\}$, where we have used the uniform upper and lower bounds of $\eta(t)$ for $\alpha>n\beta+1$, $\beta\ge1$ in Lemma \ref{boundeta}.

Next we study the case $\alpha=n\beta+1$ by the following three steps.

Step 1: Consider the function
\[Q=fu^{\alpha-1}\sigma_n^{-\beta}.
\]
Since \[(fu^\alpha\sigma_n^{-\beta})_{ij}=Q_{ij}u+Q_iu_j+Q_ju_i+Qu_{ij},
\]
we get
\begin{eqnarray}\label{estimate2}
\nonumber\partial_tQ&=&-(\alpha-1)f^2u^{2\alpha-2}\sigma_n^{-2\beta}+\eta(\alpha-1-n\beta)fu^{\alpha-1}\sigma_n^{-\beta}
\\
\nonumber&{}&+\beta f^2u^{2\alpha-1}\sigma_n^{-2\beta-1}\sum_i^n\sigma_n^{ii}+\beta fu^{\alpha-1}\sigma_n^{-\beta-1}\sigma_n^{ij}(fu^{\alpha}\sigma_n^{-\beta})_{ij}
\\
\nonumber&=&-(\alpha-1-n\beta)Q^2+\eta(\alpha-1-n\beta)Q+\beta fu^{\alpha}\sigma_n^{-\beta-1}\sigma_n^{ij}Q_{ij}\\
\nonumber&{}&+2\beta fu^{\alpha-1}\sigma_n^{-\beta-1}\sigma_n^{ij}Q_iu_j
\\
\nonumber&=&\beta fu^{\alpha}\sigma_n^{-\beta-1}\sigma_n^{ij}Q_{ij}+2\beta fu^{\alpha-1}\sigma_n^{-\beta-1}\sigma_n^{ij}Q_iu_j
\end{eqnarray}
It is easy to see
\begin{equation}\label{usigma}
C^{-1}\le Q \le C,
\end{equation} where $C$ depends only on the initial hypersurface.

Step 2:
Let $w=\log u$. Then
\[
h_{ij}=u_{ij}+u\delta_{ij}=u(w_{ij}+w_iw_j+\delta_{ij})
\]
We may prove $|\nabla w|< A$, for some positive constant $A>0$ along the flow. Otherwise there is a point $(x_{t_0}, t_0)$ where $t_0$ is the first time, such that $|\nabla u|^2-Au^2=0$, $A>0$ is a constant to be determined later. Hence at the point $(x_{t_0},t_0)$, $\nabla_i|\nabla w|^2=0$ and $\partial_t |\nabla w|^2\ge 0$.
Choosing an orthonormal frame and rotating the the coordinates, such that $w_1=|\nabla w|$, $w_i=0$ for $i=2,\cdots,n$, and $(w_{ij})$ is diagonal at $(x_{t_0}, t_0)$. We then get
\[(a_{ij}):=(w_{ij}+w_iw_j+\delta_{ij})= {\rm{diag}} (1+w_1^2,1+w_{22},\cdots,1+w_{nn}),
\]
and
\begin{eqnarray*}
 0&\le&\partial_t(|\nabla u|^2-Au^2)=-2u_i(fu^\alpha)_i\sigma_n^{-\beta}+2\beta fu^\alpha\sigma_n^{-\beta-1}\sigma_n^{kl}\nabla_ih_{kl}u_i+2Afu^{\alpha+1}\sigma_n^{-\beta}
 \\&\le&-2u_i(fu^\alpha)_i\sigma_n^{-\beta}+4n\beta fu^{\alpha+1}\sigma_n^{-\beta}-2\beta fu^{\alpha+2}\sigma_n^{-\beta-1}\sum_i^n\sigma_n^{ii}-2\beta  fu^\alpha\sigma_n^{-\beta-1}\sigma_n^{ij}h_{li}h_{lj}
 \\&{}&+2A(n\beta+1)fu^{\alpha+1}\sigma_n^{-\beta}-2A\beta fu^{\alpha+2}\sigma_n^{-\beta-1}\sum_i^n\sigma_n^{ii}
 \\&{}&+2A\beta fu^{\alpha}\sigma_n^{-\beta-1}\sigma_n^{ij}u_iu_j+2\beta fu^\alpha\sigma_n^{-\beta-1}\sigma_n^{ij}u_iu_j.
\end{eqnarray*}
Substituting $u_i=uw_i$ and  $w_1^2=A$ into the above inequality, denoting that  $\sigma_n=\sigma_n(a_{ij})$, we have
\begin{eqnarray*}
 0&\le& A(1+n\beta-\alpha)+\sqrt{A}\frac{|\nabla f|}{f}+2n\beta-(A+1)^2\beta\frac{\sigma_n^{11}}{\sigma_n}\\&{}&-(A+1)\beta\frac{\sum_i^n\sigma_n^{ii}}{\sigma_n}+ A(A+1)\beta\frac{\sigma_n^{11}}{\sigma_n}
 \\&\le&\sqrt{A}\frac{|\nabla f|}{f}+2n\beta-(A+1)\beta\frac{\sum_i^n\sigma_n^{ii}}{\sigma_n}.
\end{eqnarray*}
Then $(A+1)C_0 \sigma_n^{-\frac{1}{n}}\le (A+1)\beta\frac{\sum_i^n\sigma_n^{ii}}{\sigma_n}\le\sqrt{A}\frac{|\nabla f|}{f}+2n\beta$, since $\frac{\sum_i^n\sigma_n^{ii}}{\sigma_n}\ge C(n)\sigma_n^{-\frac{1}{n}}$ by the classic Newton-MacLaurin inequality, and  $\sigma_n^{-\frac{1}{n}}(a_{ij})=u\sigma_n^{-\frac{1}{n}}[W_u]$ is bounded by \eqref{usigma}.
Let $A$ be large enough, we then get a contradiction. Hence we obtain
\begin{equation}\label{nablaubound}
|\nabla\log u|\le C.
\end{equation}

Step 3: For the normalised flow \eqref{NSF}, $\int_{S^n}u\sigma_n dx=|S^n|$ is constant. By Step 1, there is a positive constant $C$, such that $C^{-1}\le u^{n\beta}\sigma_n^{-\beta}\le C$. Hence we have
\[C^{-\frac{1}{\beta}}u_{\min}^{n+1}(t)\le\frac{\int_{S^n}u\sigma_ndx}{|S^n|}=1\le C^{\frac{1}{\beta}} u_{\max}^{n+1}(t)\]
We therefore obtain the uniform upper and lower bounds on $u$ from \eqref{nablaubound}.
\end{proof}
Since $\frac{1}{C}\le u\le C$, for some positive constant $C$, by the convexity of the hypersurface \eqref{nablau}, it is easy to get the following gradient estimate.
\begin{cor}\label{C1estimate}
Let $u(\cdot,t)$ be a solution to the flow \eqref{NSF}. Then we have the gradient estimate
\[|\nabla u(\cdot,t)|\le C,\]
where the positive constant $C$ depends only on $\alpha$, $\beta$, $f$ and the initial hypersurface.
\end{cor}

Similarly we have the estimates for the radial function $r$.
\begin{lem}\label{rC0estimate,rC1estimate}
Let $\mathcal{M}_t$ be the solution to the flow \eqref{NSF}. Then we have the estimate
\[\mathop{\min}_{S^n\times(0,T]}u\le r(\cdot,t)\le \mathop{\max}_{S^n\times(0,T]}u,\]
and
\[|\nabla r(\cdot,t)|\le C,\]
where  $C>0$ depends only on $\alpha$, $\beta$, $f$ and the initial hypersurface.
\end{lem}
\begin{proof}
By \eqref{nablau} and \eqref{nablar}, we infer that
\[\mathop{\min}_{S^n}u(\cdot,t)=\mathop{\min}_{S^n}r(\cdot,t),\quad \mathop{\max}_{S^n}u(\cdot,t)=\mathop{\max}_{S^n}r(\cdot,t) \quad {\rm{and}} \quad|\nabla r|\le\frac{r^2}{u}.\]
Therefore, the two estimates follow from  Lemmata \ref{C0estimate1}-\ref{C0estimate2} directly.
\end{proof}

\begin{lem}\label{lowerboundsigma_n}
Let $X(\cdot, t)$ be a uniformly convex solution to the normalised flow \eqref{NSF} which encloses the origin for $t\in [0,T)$.  Then there is a positive constant $C$ depending only on $f$, $\alpha,\beta$ and the initial hypersurface, such that
\[\sigma_n([w_u])\ge C.\]
\end{lem}
\begin{proof}
Consider the following auxiliary function
\[G=\frac{-u_t+\eta u}{u-\epsilon}={\frac{ fu^\alpha{\sigma}_{n}^{-\beta}}{u-\epsilon}},\]
where
 ${\epsilon}=\frac{1}{2}\min_{{S^n\times{[0,T)}}}u$.
Suppose that $G_{\max}(t)=\max_{x\in S^n}G(x,t)=G(x_t,t)$, at $x_t$, we then have
\begin{eqnarray}\label{q_1}
0=G_i=\frac{-u_{ti}+\eta u_i}{u-\epsilon}-\frac{(-u_t+ \eta u){u_i}}{(u-\epsilon)^2},
\end{eqnarray}
\begin{eqnarray}\label{q_2}
0{\ge}{G_{ij}}=\frac{-u_{tij}+ \eta u_{ij}}{u-\epsilon}-\frac{(-u_t+ \eta u){u_{ij}}}{(u-\epsilon)^2},
\end{eqnarray}
and
\begin{eqnarray*}
{\partial}_{t}{G}&=&\frac{-u_{tt}+\eta_tu+\eta u_{t}}{u-\epsilon}-\frac{(-u_t+\eta u){u_t}}{({u-\epsilon})^2}
\\&=&\frac{\alpha fu^{\alpha-1}u_t \sigma_n^{-\beta}-\beta fu^{\alpha}\sigma_n^{-\beta-1} \sigma_n^{ij}(u_{tij}+u_t\delta_{ij})}{u-\epsilon} -G\frac{u_t}{u-\epsilon}
\\&\le&\alpha Gu_tu^{-1}+\frac{\beta fu^\alpha\sigma_n^{-\beta-1}  \sigma_{n}^{ij}(Gu_{ij}-\eta u_{ij}-u_{t}\delta_{ij})}{u-\epsilon}-G\frac{u_t}{u-\epsilon}
\\&\le&\alpha G\big(\eta-G\frac{u-\epsilon}{u}\big)+(G-\eta)\frac{\beta fu^\alpha\sigma_n^{-\beta-1} \sigma_n^{ij}(h_{ij}-u\delta_{ij})}{u-\epsilon}\nonumber\\&{}&+\beta fu^\alpha\sigma_n^{-\beta-1}\sum_i\sigma_n^{ii}(G-\frac{ \eta u}{u-\epsilon})+G(G-\frac{\eta u}{u-\epsilon})
\\&=& (n\beta+1-\alpha+\frac{\epsilon\alpha}{u})G^2+\eta(\alpha-n\beta-\frac{u}{u-\epsilon})G-\epsilon\beta G^2\frac{\sum_i\sigma_n^{ii}}{\sigma_n}.
\end{eqnarray*}
Without loss of generality, we assume $G\gg1$.

For the case $\alpha\ge0$, $\beta>1$; $\alpha>1-n\beta-2\beta$, $\alpha\ne1-\beta$, $0<\beta\le1$ and $\alpha=0$, $\beta=1$, we have $\eta(t)\le C$ for some positive constant $C$ by Lemma \ref{upboundeta}.
Applying $G<G^2$ and the inequality $\frac{\sum_i\sigma_n^{ii}}{\sigma_n}\ge C\sigma_n^{-\frac{1}{n}}$, we get
\begin{eqnarray*}
{\partial}_{t}{G}&\le& \left | \left(n\beta+1-\alpha+\frac{\epsilon\alpha}{u}\right)\right|G^2+\eta\left |\left(\alpha-n\beta-\frac{u}{u-\epsilon}\right)\right |G^2-\epsilon\beta G^2\frac{\sum_i\sigma_n^{ii}}{\sigma_n}
\\&\le& C_1G^2-C_2G^2G^\frac{1}{\beta n}.
\end{eqnarray*}
For the case $\alpha<0$, $\beta>1$, we obtain $\eta(t)=\frac{\int_{S^n}fu^\alpha\sigma_n^{1-\beta}dx}{|S^n|}\le C_0G(x_t,t)^\frac{\beta-1}{\beta}$ since $G=\frac{fu^\alpha\sigma_n^{-\beta}}{u-\epsilon}$ and $u$ is uniformly bounded. Applying $G^\frac{\beta-1}{\beta}<G$ at $x_t$ and the inequality $\frac{\sum_i\sigma_n^{ii}}{\sigma_n}\ge C\sigma_n^{-\frac{1}{n}}$, we get
\begin{eqnarray*}
{\partial}_{t}{G}&\le& \left | \left(n\beta+1-\alpha+\frac{\epsilon\alpha}{u}\right)\right|G^2+\left |\left(\alpha-n\beta-\frac{u}{u-\epsilon}\right)\right |G^2-\epsilon\beta G^2\frac{\sum_i\sigma_n^{ii}}{\sigma_n}
\\&\le& C_1G^2-C_2G^2G^\frac{1}{\beta n}.
\end{eqnarray*}
It is easy to see that there exists a positive constant $C_3$ , s.t. $G\le C_3$, where $C_3$ is a constant depending only on $f$, $\alpha$, $\beta$ and the initial hypersurface.
Hence we obtain $\sigma_n([w_u])\ge C$, where $C$ is a constant depending only on $f$, $\alpha$, $\beta$ and the initial hypersurface.
\end{proof}
Hence we get $\eta(t)\le C$ for the case $\alpha>1-n\beta-2\beta$, $\alpha\ne1-\beta$, $\beta>0$ and $\alpha=0$, $\beta=1$ by Lemma \ref{upboundeta} and Lemma \ref{lowerboundsigma_n}.
Next we prove the principal curvature radii of $\mathcal{M}_t$ is bounded. We study an expanding flow of Gauss curvature for the dual hypersurface of $\mathcal{M}_t$. The method is inspired by \cite{LSW16}. Similar idea was previously used by Ivaki in \cite{Iva14}.

Under the evolution equation \eqref{NSF},  the radial function of the hypersurface $\mathcal{M}$ evolves as
\begin{equation}\label{NSFr}
\left\{\begin{array}{ll}
\partial_t{r}(\xi,t)&=-f{\frac{\sqrt{r^2+|{\nabla}r|^2}}{r}}u^\alpha K^\beta+\eta r\\[0.2cm]
r(.,0)&=r_0
\end{array}\right.
\end{equation}
where $K$ is the Gauss-Kronecker curvature of $\mathcal{M}$.

Let $\Omega$ be a convex body enclosing the origin, $\partial\Omega=\mathcal{M}$. The dual body of $\Omega$ with respect to the origin, denoted by $\Omega^*$.
Its support function $u^*(\xi,t)=\frac{1}{r(\xi,t)}$, hence $\hbar^*_{ij}=\nabla_{ij}\frac{1}{r}+\frac{1}{r}\delta_{ij}=\frac{-rr_{ij}+2r_ir_j+r^2\delta_{ij}}{r^3}$ and\begin{equation}\label{relationeqs}
K=\frac{\det {\hbar}_{ij}}{\det g_{ij}},\qquad \frac{1}{\sigma_n^*}=\frac{\det e_{ij}}{\det {\hbar}^*_{ij}},\qquad\frac{\det e_{ij}}{\det g_{ij}}=\frac{1}{r^{2n-2}(r^2+|{\nabla}r|^2)}.
\end{equation}
Hence by \eqref{nablar} and \eqref{relationeqs}, we obtain the following equality
\begin{eqnarray}\label{polarrelation}
\frac{{u(x,t)}^{n+2}{u^*(\xi,t)}^{n+2}}{K(p)K^*(p^*)}=1
\end{eqnarray}
where $p\in\mathcal{M}_t$, $p^*$ satisfies the polar relation $p\cdot p^*=1$ and $p^*\in{\mathcal{M}^*_t}$, $K^*$ is the Gauss curvature at $p^*$. $x,\, \xi$ are the unit outer normals of $\mathcal{M}_t$ and ${\mathcal{M}^*_t}$ respectively.
Therefore, by the normalised flow \eqref{NSFr} and the relation \eqref{polarrelation},  we obtain the flow for the support function $u^*$
\begin{eqnarray*}
{\partial_t{u^*(\xi,t)}}&=&\partial_t{\frac{1}{r(\xi,t)}}
\\
&=&f(x)(u^*(\xi,t))^{1+\beta n+2\beta}(r^*(x,t))^{1-\alpha -2\beta-\beta n}(K^*)^{-\beta}-\eta u^*(\xi,t)
\\
&=&f(x)(u^*(\xi,t))^{1+\beta n+2\beta}(r^*(x,t))^{1-\alpha -2\beta-\beta n}(\sigma_n^*[w_{u^*}])^{\beta}-\eta u^*(\xi,t).
\end{eqnarray*}
where $r^*(x,t)=\sqrt{(u^*)^2+|\nabla u^*|^2}(\xi,t)$, $\sigma_n^*=\sigma_n[W_{u^*} ]$ and $W_{u^*}=\nabla^2u^*+uI$.\\
By Lemma \ref{rC0estimate,rC1estimate}, $\frac{1}{C}\le u^*\le C$ and $|\nabla u^*|\le C$ for some $C$ only depending on the initial hypersurface.

\begin{lem}\label{boundsprincipleradii}
Let $X(\cdot, t)$ be the solution to the normalised flow \eqref{NSF} which encloses the origin. Then there is a constant $C$ depending only on the initial hypersurface and $f$, $\alpha$, $\beta$, such that the principal curvature radii of $X(\cdot, t)$ are bounded from above and below
\[C^{-1}\le\lambda_i(.,t)\le C\] for all $t\in[0,T)$ and $i=1,...,n$.
\end{lem}
\begin{proof}
Let $h^*_{ij}=u^*_{ij}+u^*\delta_{ij}$, and $h_*^{ij}$ be the inverse matrix of $h^*_{ij}$.
Consider the auxiliary function
\[w(\xi, t, \tau)=\log h_*^{\tau\tau} -\varepsilon \log u^*+\frac{M}{2}({u^*}^2+|\nabla u^*|^2),\]
where
$\tau$ is a unit vector in the tangential space of $S^n$, while $\epsilon $ and $M$ are large constants to be decided. Assume $w$ achieve its maximum at $(\xi_0,t_0)$ in the direction $\tau=(1, 0, \cdots, 0)$. By a coordinate rotation, $h^*_{ij}$ and $h_*^{ij}$ are diagonal at this point. Then at the point $(\xi_0,t_0)$.
\[w=\log h_*^{11}-\varepsilon  \log u^*+\frac{M}{2}({u^*}^2+|\nabla u^*|^2),\]
\begin{eqnarray}\label{partial1}
0=\nabla_iw=-h_*^{11}\nabla_ih^*_{11}-\varepsilon \frac{\nabla_{i}u^*}{u^*}+Mu^*u^*_{i}+M\nabla_{k}u^*\nabla_{ki}u^*,
\end{eqnarray}
\begin{eqnarray}\label{partial2}
0\ge\nabla_{ij}w&=&-h_*^{11}\nabla_{ij}h^*_{11}+2h_*^{11}h_*^{kk}\nabla_{1}h^*_{ik}\nabla_{1}h^*_{kj}-(h_*^{11})^2\nabla_ih^*_{11}\nabla_jh^*_{11}
-\varepsilon \frac{\nabla_{ij}u^*}{u^*}
\nonumber\\&{}&+\varepsilon \frac{\nabla_{i}u^*\nabla_{j}u^*}{{u^*}^2}+Mu^*_iu^*_j+Mu^*u^*_{ij}+M\nabla_{ki}u^*\nabla_{kj}u^*+M\nabla_ku^*\nabla_{kij}u^*
\end{eqnarray}
Set $\Phi=f(u^*)^{1+\beta n+2\beta}(r^*)^{1-\alpha -2\beta-\beta n}$, we have
\begin{eqnarray}
0\le\partial_tw&=&-h_*^{11}\partial_th^*_{11}-\varepsilon \frac{\partial_tu^*}{u^*}+Mu^*\nabla_{t}u^*+M\nabla_{k}u^*\nabla_{kt}u^*
\nonumber\\
&=&-h_*^{11}\Big(\Phi_{11}{\sigma^*_n}^\beta+2\beta{\sigma^*_n}^{\beta-1}\nabla_1\Phi\nabla_1\sigma^*_n+\beta(\beta-1)\Phi{\sigma^*_n}^{\beta-2}(\nabla_1\sigma^*_n)^2
\nonumber\\
&{}&+\beta\Phi{\sigma^*_n}^{\beta-1}\nabla_{11}\sigma^*_n
-\eta h^*_{11}+\Phi{\sigma^*_n}^{\beta}\Big)-\varepsilon \frac{\Phi {\sigma^*_n}^{\beta}-\eta u^*}{u^*}
\nonumber\\
&{}&+Mu^*(\Phi {\sigma^*_n}^{\beta}-\eta u^*)+Mu^*_k(\Phi {\sigma^*_n}^{\beta}-\eta u^*)_k
\nonumber\\
&=&-h_*^{11}\Big(\Phi_{11}{\sigma^*_n}^\beta+2\beta{\sigma^*_n}^{\beta}\nabla_1\Phi h_*^{ij}\nabla_1h^*_{ij}+\beta(\beta-1)\Phi{\sigma^*_n}^{\beta}(h_*^{ij}\nabla_1h^*_{ij})^2
\nonumber\\
&{}&+\beta\Phi{\sigma^*_n}^{\beta}\big(h_*^{ij}\nabla_{ij}h^*_{11}+n-h^*_{11}\sum_ih_*^{ii}-h_*^{ii}h_*^{jj}(\nabla_1h^*_{ij})^2+(h_*^{ij}\nabla_1h^*_{ij})^2\big)
\nonumber\\
&{}&-\eta h^*_{11}+\Phi{\sigma^*_n}^{\beta}\Big)
-\varepsilon \frac{\Phi {\sigma^*_n}^{\beta}-\eta u^*}{u^*}+Mu^*(\Phi {\sigma^*_n}^{\beta}-\eta u^*)+Mu^*_k(\Phi {\sigma^*_n}^{\beta}-\eta u^*)_k
\nonumber
\end{eqnarray}
By \eqref{partial2} and multiplying $\Phi^{-1}{\sigma^*_n}^{-\beta}$ the two sides of the above inequality,we obtain
\begin{eqnarray*}
0&\le&-h_*^{11}\frac{\nabla_{11}\Phi}{\Phi}+h_*^{11}\big((\frac{\nabla_1\Phi}{\Phi})^2+\beta^2(h_*^{ij}\nabla_1h^*_{ij})^2\big) -h_*^{11}\beta^2(h_*^{ij}\nabla_1h^*_{ij})^2
\nonumber\\
&{}&+\beta\Big(-2h_*^{11}h_*^{ij}h_*^{kk}\nabla_1h^*_{ik}\nabla_1h^*_{jk}
+(h_*^{11})^2h_*^{ij}\nabla_ih^*_{11}\nabla_jh^*_{11}
+n\varepsilon  {u^*}^{-1}
\nonumber\\
&{}&-\varepsilon \sum_ih_*^{ii}-\varepsilon  h_*^{ij}u^*_iu^*_j{u^*}^{-2}-Mh_*^{ij}h^*_{ik}h^*_{jk}+nMu^*
-Mh_*^{ij}u^*_k\nabla_kh^*_{ij}\Big)
\nonumber\\
&{}&-h_*^{11}(1+n\beta)+\beta\sum_ih_*^{ii}+\beta h_*^{11}h_*^{ii}h_*^{jj}(\nabla_1h^*_{ij})^2+Mu^*+Mu^*_k\frac{\nabla_k\Phi}{\Phi}
\nonumber\\&{}&+\beta M h_*^{ij}u^*_k\nabla_kh^*_{ij}
-\eta M\frac{{u^*}^2+|\nabla u^*|^2}{\Phi{\sigma^*_n}^\beta}-\varepsilon  {u^*}^{-1}+\eta\frac{\varepsilon +1}{\Phi{\sigma^*_n}^\beta}
\\&\le&-2\beta h_*^{11}h_*^{ij}h_*^{kk}\nabla_{1}h^*_{ik}\nabla_{1}h^*_{kj}+\beta(h_*^{11})^2h_*^{ij}\nabla_{i}h^*_{11}\nabla_{j}h^*_{11}+\beta h_*^{11}h_*^{ii}h_*^{jj}(\nabla_{1}h^*_{ij})^2
\nonumber\\
&{}&+\beta\frac{n\varepsilon }{u^*}
-\beta\varepsilon \sum h_*^{ii}
+\beta\sum_ih_*^{ii}+Mu^*+h_*^{11}(\frac{\nabla_1\Phi}{\Phi})^2
\nonumber\\
&{}&+M\beta nu^*-h_*^{11}\frac{\nabla_1\nabla_1\Phi}{\Phi}+ Mu^*_k\frac{\nabla_k\Phi}{\Phi}-\eta\frac{M{r^*}^2-\varepsilon -1}{\Phi {\sigma^*_n}^\beta}
\\&\le& C_0-\beta(\varepsilon -1) h_*^{11}+h_*^{11}(\frac{\nabla_1\Phi}{\Phi})^2-h_*^{11}\frac{\nabla_1\nabla_1\Phi}{\Phi} +Mu^*_k\frac{\nabla_k\Phi}{\Phi}-\eta\frac{M{r^*}^2-\varepsilon -1}{\Phi {\sigma^*_n}^\beta},
\end{eqnarray*}
where we use the Cauchy inequality $2\beta\frac{\nabla_1\Phi}{\Phi} h_*^{ij}\nabla_1h^*_{ij}\le(\frac{\nabla_1\Phi}{\Phi})^2+\beta^2(h_*^{ij}\nabla_1h^*_{ij})^2$ for the second term.\\
Since $\nabla_k\Phi=\nabla_k\big(f(u^*)^{1+\beta n+2\beta}(r^*)^{1-\alpha -2\beta-\beta n}\big)$, and it is direct to calculate
\[r^*_k=\frac{u^*u^*_k+\sum_iu^*_iu^*_{ik}}{r^*}=\frac{u^*_kh^*_{kk}}{r^*}\]
\[r^*_{kl}=\frac{u^*u^*_{kl}+u^*_ku^*_l+\sum_iu^*_iu^*_{ikl}+\sum_iu^*_{ik}u^*_{il}}{r^*}-\frac{u^*_ku^*_lh^*_{kk}h^*_{ll}}{(r^*)^3}\]
hence, by \eqref{partial1}, we obtain
\begin{eqnarray*}
& &h_*^{11}(\frac{\nabla_1\Phi}{\Phi})^2-h_*^{11}\frac{\nabla_1\nabla_1\Phi}{\Phi} +Mu^*_k\frac{\nabla_k\Phi}{\Phi}\\
&\le& Ch_*^{11}+C+C\frac{1}{h_*^{11}}+Ch_*^{11}(u^*_{11})^2+CM+Ch_*^{11}u^*_ku^*_{k11}+CM\Sigma _k{u^*_ku^*_kh^*_{kk}}
\nonumber\\&\le&Ch_*^{11}+C\frac{1}{h_*^{11}}+C+C\varepsilon+CM
\end{eqnarray*}
Choosing $M\ge\frac{\varepsilon+1}{\min {r^*}^2}$, the inequality becomes
\[0\le-\beta(\varepsilon -1) h_*^{11}+Ch_*^{11}+C\frac{1}{h_*^{11}}+C+C\varepsilon+CM\le-\beta\varepsilon h_*^{11}+Ch_*^{11}+C\varepsilon+CM\]
By choosing $\varepsilon $ large to get
\[
0\le C_1-C_2h_*^{11}.
\]
That is, $h_*^{11}\le C_3$, where $C_3$ is a constant depending only on $f$, $\alpha$, $\beta$ and the initial hypersurface. Hence the Gauss curvature of $\mathcal{M}^*_t$, $K^*\le C$.
From \eqref{polarrelation}, we know $\sigma_n=\lambda_1\cdots\lambda_n\le C_1$.
 Therefore we get the $C^2$ estimate $C^{-1}\le\lambda_i\le C, i=1,\ldots,n$ by Lemma \ref{lowerboundsigma_n} for the solutions to the normalised flow \eqref{NSF}.
\end{proof}

\section{Proof of Theorems \ref{main1}-\ref{main3}}
\begin{proof}[Proof of Theorems \ref{main1}-\ref{main2}]
From the estimates obtained in Lemma \ref{boundsprincipleradii}, we know that the equations \eqref{NSF} are uniformly parabolic. By the $C^0$ estimates Lemmas \ref{C0estimate1} and Lemmas \ref{C0estimate2}, the gradient estimates (Lemma \ref{C1estimate}) the $C^2$ estimates Lammas \ref{boundsprincipleradii}, and the Krylov's theory \cite{Knv87}, we get the H\"{o}lder continuity of $\nabla^2u$ and $u_t$. Then we can get higher order derivation estimates by the regularity theory of the uniformly parabolic equations. Therefore we get the long time existence and the uniqueness of the smooth solution to the normalized flows \eqref{NSF}. Recall the Lemma \ref{monotone12}, Lemma \ref{monotone3} and Lemma \ref{monotone4}, we complete the proof.
\end{proof}
\begin{proof}[Proof of Theorem \ref{main3}]
For the case $\alpha\ge1+n\beta$, $\beta>0$.
To complete the proof of Theorem \ref{main3}, it suffices to show that the solution of \eqref{elliptic eq} is unique.

Case 1: $\alpha>1+n\beta$.
Let $u_1$ ,$u_2$ be two smooth solutions of \eqref{elliptic eq}, i.e. \\
\[fu_1^{\alpha-1}\sigma_n^{-\beta}(\nabla^2u_1+u_1I)=c , fu_2^{\alpha-1}\sigma_n^{-\beta}(\nabla^2u_2+u_2I)=c.\]
Suppose $M=\frac{u_1}{u_2}$ attains its maximum at $X_0\in S^n$, then at $x_0$,
\[0=\nabla{\log M}=\frac{\nabla u_1}{u_1}-\frac{\nabla u_2}{u_2},\]
\[0\ge\nabla^2{\log M}=\frac{\nabla^2u_1}{u_1}-\frac{\nabla^2u_2}{u_2}.\]
Hence at $x_0$, we get
\[1=\frac{u_1^{\alpha-1}\sigma_n^{-\beta}(\nabla^2u_1+u_1I)}{u_2^{\alpha-1}\sigma_n^{-\beta}(\nabla^2u_2+u_2I)}=\frac{u_1^{\alpha-1-n\beta}\sigma_n^{-\beta}(\frac{\nabla^2u_1}{u_1}+I)}{u_2^{\alpha-1-n\beta}\sigma_n^{-\beta}(\frac{\nabla^2u_2}{u_2}+I)}\ge M^{\alpha-1-n\beta}.\]
Since $\alpha>1+n\beta$, $M(x_0)=\max_{S^n} M\le1$. Similarly one can show $\min_{S^n} M\ge1$. Therefore $u_1\equiv u_2$.

Case 2: $\alpha=1+n\beta$, the elliptic equation \eqref{elliptic eq} can be written as $f^\frac{1}{\beta}u^n\sigma_n^{-1}=c$, the uniqueness of the solution has been proved in \cite{Luk93}, and Chou-Wang \cite{ChWang06} also provide a method to get the solutions for equation differ only by a dilation. We omit the proof here. Hence we complete the proof of Theorem \ref{main3}.
\end{proof}

\begin{proposition}\label{alpha}
For $1<\alpha<1+n\beta$, the solution of \eqref{elliptic eq} is unique.
\end{proposition}
\begin{proof}
Let $u_1$ ,$u_2$ be two smooth solutions of \eqref{elliptic eq},  i.e.
\[fu_1^{\alpha-1}\sigma_n^{-\beta}(\nabla^2u_1+u_1I)=c , fu_2^{\alpha-1}\sigma_n^{-\beta}(\nabla^2u_2+u_2I)=c.\]

Using the same argument in \cite{GuX18}, by the Alexandrov-Fenchel inequality in Lemma \ref{AFI}, we have
\begin{eqnarray*}
&{}&\int_{S^n}u_1u_2^{\frac{\alpha-1}{\beta}}(cf^{-1})^{-\frac{1}{\beta}}dx=\int_{S^n}u_1\sigma_n(\nabla^2u_2+u_2I)dx
=V_{n+1}(u_1,u_2,\cdots,u_2)
\\&\ge&V_{n+1}(u_2,u_2,\cdots,u_2)^{\frac{n}{n+1}}V_{n+1}(u_1,u_1,\cdots,u_1)^{\frac{1}{n+1}}
\\&=&\Big(\int_{S^n}(cf^{-1})^{-\frac{1}{\beta}}u_2^{\frac{\alpha-1+\beta}{\beta}}dx\Big)^\frac{n}{n+1}\Big(\int_{S^n}(cf^{-1})^{-\frac{1}{\beta}}u_1^{\frac{\alpha-1+\beta}{\beta}}dx\Big)^\frac{1}{n+1}.
\end{eqnarray*}
On the other hand, H\"{o}lder inequality gives
\[\int_{S^n}u_1u_2^{\frac{\alpha-1}{\beta}}(cf^{-1})^{-\frac{1}{\beta}}dx\le\Big(\int_{S^n}(cf^{-1})^{-\frac{1}{\beta}}u_1^{\frac{\alpha-1+\beta}{\beta}}dx\Big)^\frac{\beta}{\alpha-1+\beta}\Big(\int_{S^n}(cf^{-1})^{-\frac{1}{\beta}}u_2^{\frac{\alpha-1+\beta}{\beta}}dx\Big)^\frac{\alpha-1}{\alpha-1+\beta}.
\]
Combining the above two inequalities, for $1<\alpha<1+n\beta$, $\beta\ge1$, we have
\[\int_{S^n}(cf^{-1})^{-\frac{1}{\beta}}u_1^{\frac{\alpha-1+\beta}{\beta}}dx \ge \int_{S^n}(cf^{-1})^{-\frac{1}{\beta}}u_2^{\frac{\alpha-1+\beta}{\beta}}dx. \]
Similar argument by interchanging the role of $u_1$ and $u_2$ gives
\[\int_{S^n}(cf^{-1})^{-\frac{1}{\beta}}u_2^{\frac{\alpha-1+\beta}{\beta}}dx \ge \int_{S^n}(cf^{-1})^{-\frac{1}{\beta}}u_1^{\frac{\alpha-1+\beta}{\beta}}dx \]
Therefore all the above inequalities are equalities. Using the equality condition in the Alexandrov-Fenchel inequality in Lemma \ref{AFI}, we have $u_1\equiv u_2$.
\end{proof}

\vskip10pt



\begin{thebibliography}{99}

\bibitem{And94}          Andrews, B.:
                                   Contraction of convex hypersurfaces in Euclidean space.
                                   Calc. Var. Partial Differential Equations 2 (1994), no. 2, 151--171.
\bibitem{And94(1)}       Andrews, B.:
                                   Entropy estimates for evolving hypersurfaces.
                                   Communications in Analysis and Geometry, (1994), 2(1),267-275.
\bibitem{And97}
                         Andrews, B.:
                                   Monotone quantities and unique limits for evolving convex hypersurfaces. International Mathematics Research Notices,(1997), 20(20),1001-1031.

\bibitem{And96}           Andrews, B.:
                                  Contraction of convex hypersurfaces by their affine normal.
                                  J. Differential Geom., 43 (1996), 207-230.

\bibitem{And99}           Andrews, B.:
                                    Gauss curvature flow: the fate of the rolling stones.
                                    Invent. Math., 138 (1999), no. 1, 151--161.

\bibitem{And00}           Andrew, B.:
                                    Motion of hypersurfaces by Gauss curvature.
                                    Pacific J. Math., 195(2000), 1-34.

\bibitem{AGN16}             Andrews, B.; Guan, Pengfei.;  Ni, Lei.:
                                    Flow by powers of the Gauss curvature.
                                    Adv. Math., 299(2016), 174-201.

\bibitem{BCD16}          Brendle, S.; Choi, K.; Daskalopoulos, P.:
                                     Asymptotic behavior of flows by powers of the Gauss curvature.,
                                     Acta Mathematica., 219(2016).

\bibitem{BIS16}          Bryan, P.; Ivaki, M.; Scheuer, J.:
                                     A unified flow approach to smooth, even $L_p$-Minkowski problems.
                                     Analysis and PDE., 12(2019),259-280.

\bibitem{BLDZ13}         B\"{o}r\"{o}czky, J.; Lutwak, E.;
                             D. Yang; G.Zhang.:
                                       The logarithmic Minkowski problem.
                                  J. Amer. Math. Soc.26 (2013), 831-852.

\bibitem{C06}             Chen, W.:
                                  $L_p$ Minkowski problem with not necessarily positive data.
                                     Adv. Math.201 (2006), pp. 77-89.

\bibitem{CLZ17}             Chen, S.; Li, Q.-R.; Zhu, G.:
                                On the Lp Monge-Amp$\grave{e}$re equation.
                                Journal of Differential Equations. 263(2017), 4997-5011.

\bibitem{CY77}             Cheng, S.-Y.; Yau, S.-T.:
                                  On the regularity of the n-dimensional Minkowski problem.
                                  Comm. Pure Appl. Math. 20(1977), 41-68.

\bibitem{ChoiDask16}       Choi, K.; Daskalopoulos, P.:
                                     Uniqueness of closed self-similar solutions to the Gauss curvature flow.
                                     ArXiv:1609.05487.

\bibitem{ChWang00}         Chou, K.-S.; Wang, X.-J.:
                                    A logarithmic Gauss curvature flow and the Minkowski problem.
                                    Annales De Linstitut Henri Poincare, 2000, 17(6):733-751.

\bibitem{ChWang06}          Chou, K.-S.; Wang, X.-J.:
                                   The $L_p$ Minkowski problem and the Minkowski problem in centroaffine geometry.
                                   Adv. in Math., 205 (2006), 33-83.


\bibitem{Chow85}         Chow, B.:
                                    Deforming convex hypersurfaces by the $n$-th root of the Gaussian curvature.
                                    J. Differential Geom. 22 (1985), no. 1, 117--138.

\bibitem{Chow87}         Chow, B.:
                                   Deforming convex hypersurfaces by the square root of the scalar curvature.
                                   Invent. Math. 87 (1987), no. 1, 63--82.

\bibitem{ChowH97}          Chow, B.; Tsai, D. H. :
                                   Expansion of convex hypersurfaces by nonhomogeneous functions of curvature.
                                   Asian J. Math. 1(1997):769-784.


\bibitem{Fir74}             Firey, W. J.:
                                     Shapes of worn stones.
                                     Mathematika. 21 (1974), 1--11.


\bibitem{Gerh14}            Gerhardt, C.:
                                     Non-scale-invariant inverse curvature flows in Euclidean space.
                                     Car. Var. Partial Differential Equations. 49(2014):471-489.

\bibitem{GuNi17}            Guan, P. F.; Ni, L.:
                                       Entropy and a convergence theorem for Gauss curvature flow in high dimensions,
                                       J. Eur. Math. Soc. 19(2017), pp. 3735-3761.

\bibitem{GuX18}             Guan, P. F.; Xia, C.:
                                       $L^p$ Christoffel-Minkowski problem: the case $1<p<k+1$.
                                       Cal. Var. Partial Differential Equations. 57(2018), no.2, Art.69, 23pp.

\bibitem{CEDZ10}       Haberl, C.; Lutwak, E.; Yang, D.; Zhang, G.:
                                   The even Orlicz Minkowski problem.
                                  Adv. in Math., 224(2010),2485-2510.

\bibitem{HLW16}             He, Y.; Li, Q.-R.; Wang, X.-J.:
                                   Multiple solutions of the $L_p$-Minkowski problem.
                                   Calculus of Variations and Partial Differential Equations, 55(5), (2016),117.

\bibitem{H94}               H\"{o}rmander, L.:
                                      Notions of Convexity, Birkhauser, Boston,(1994).

\bibitem{HLX15}            Huang Y.; Liu J-K.; Wang X-J.:
                                    On the uniqueness of $L_p$-Minkowski problem: the constant $p$-curvature case in $R^3$.
                                    Adv. Math. 281(2015), pp.906-927.

\bibitem{HLYZ16}           Huang, Y.; Lutwak, E.; Yang, D.; Zhang, G.:
                                       Geometric measures in the dual  Brunn-Minkowski theory
                                       and their associated Minkowski problems.
                                        Acta Math. 216 (2016), no. 2, 325--388.

\bibitem{HLYZ05}        Hug, D.; Lutwak, E.; Yang, D.; Zhang, G.:
                                 On the $L_p$ Minkowski problem for polytopes.
                                 Discrete Comput. Geom., 33 (2005), pp. 699-715.

\bibitem{Hui84}              Huisken, G.
                                      Flow by mean curvature of convex surfaces into spheres.
                                      J. Differential Geom. 20 (1984), no. 1, 237--266.

\bibitem{Iva14}              Ivaki, M.N.:
                                      An application of dual convex bodies to the inverse Gauss curvature flow.
                                      Proc. Amer. Math. Soc.,
                                      143(2014),no.3, pp.1257-1271.

\bibitem{Iva2018}            Ivaki, M.N.:
                                      Deforming a hypersurface by principal radii of curvature and support function. Car. Var., 58(2019). https://doi.org/10.1007/s00526-018-1462-3.


\bibitem{JLW15}           Jian, H.; Lu, J.; Wang, X.-J.:
                                      Nonuniqueness of solutions to the Lp-Minkowski problem.
                                     Advances in Mathematics, 281(2015), pp.845-856,

\bibitem{Knv87}              Krylov, N.V.:
                                      Nonlinear elliptic and parabolic equations of the second order.
                                      D.Reidel Publishing Co.,Dordrecht,1987.xiv+462 pp.

\bibitem{L17}              Li, Qi-Rui.:
                                      Infinitely many solutions for centro-affine Minkowski problem.
                                      International Mathematics Research Notices. No.00(2017), pp.1-20.

\bibitem{LSW16}        Li, Q.-R.; Sheng, W.M.; Wang, X.-J.:
                                   Flow by Gauss curvature to the Aleksandrov and dual Minkowski problems.
                                   J. Eur. Math. Soc., 22(2020), 893-923.

\bibitem{LSW18}        Li, Q.-R.; Sheng, W.M.; Wang, X.-J.:
                                   Asymptotic convergence for a class of fully nonlinear curvature flows.
                                   The Journal of Geometric Analysis.
                                   https://doi.org/10.1007/s12220-019-00169-4.

\bibitem{LW13}             Lu, J.; Wang, X.-J.:
                                     Rotationally symmetric solution to the $L_p$-Minkowski problem.
                                      J. Differential Equations, 254 (2013), pp. 983-1005.

\bibitem{Luk93}            Lutwak, E.:
                                       The Brunn-Minkowski-Firey theory. I. Mixed volumes and the Minkowski problem.
                                       J. Differential Geom. 38 (1993), no. 1, 131-150.

\bibitem{LO95}             Lutwak, E; Oliker,V.:
                                     On the regularity of solutions to a generalization of the Minkowski problem. J. Differential Geom. 41 (1995), no. 1, 227-246.


\bibitem{P78}             Pogorelov. A.V.:
                                   The multidimensional Minkowski problem.
                                   Wiley, New York, 1978.

\bibitem{Sch14}           Schneider, R.:
                                   Convex bodies: the Brunn-Minkowski theory.
                                   Cambridge University Press, Cambridge 2014.

\bibitem{Urb91}              Urbas, J.:
                                       An expansion of convex hypersurfaces.
                                       J. Differential Geom. 33 (1991), no. 1, 91--125.

\bibitem{Wang96}           Wang, X.-J.:
                                       Existence of convex hypersurfaces with prescribed Gauss-Kronecker curvature.
                                       Trans. Amer. Math. Soc. 348 (1996), 4501--4524.


\bibitem{X16}             Xia, C.:
                                Inverse anisotropic curvature flow from convex hypersurfaces.
                                J Geom Anal. 27(2016), no. 3, 1--24.


\bibitem{Zhu15}
                           Zhu, G.:
                                    The $L_p$ Minkowski problem for polytopes for $0<p<1$.
                                    Journal of Functional Analysis. 269(2015),1070-1094.

\end{thebibliography}
\end{document}